\documentclass[10pt]{article}
\title{Characterizations, Dynamical Systems and Gradient Methods for 
Strongly Quasiconvex Functions}
\author{Felipe Lara\thanks{Instituto de Alta investigaci\'on (IAI),
Universidad de Tarapac\'a, Arica, Chile. E-mail:
felipelaraobreque@gmail.com; flarao@academicos.uta.cl. Web:
felipelara.cl, ORCID-ID: 0000-0002-9965-0921} \and
Ra\'ul T. Marcavillaca \thanks{Center for Mathematical Modeling (CMM),
Universidad de Chile, Santiago, Chile. E-mail: raultm.rt@gmail.com; rtintaya@dim.uchile.cl,
ORCID-ID: 0000-0003-3748-0768}
\and
Phan T. Vuong \thanks{School of Mathematical Sciences, University of
Southampton, SO17 1BJ, Southampton, United Kingdom.
E-mail: t.v.phan@soton.ac.uk, ORCID-ID: 0000-0002-1474-994X}}

\usepackage{amsmath}
\usepackage{amsthm}
\usepackage{amssymb}
\usepackage{multicol}
\usepackage[active]{srcltx}
\usepackage{algorithm}
\usepackage{array}
\usepackage[dvipsnames,svgnames]{xcolor}
\usepackage[notcite,notref]{showkeys}
\usepackage{tikz}
\usepackage{subfig}
\usepackage{graphicx}
\usepackage{xcolor}
\usepackage{ulem}
\usepackage{amsmath}
\usepackage{amsthm}
\usepackage[utf8]{inputenc}
 \usepackage[T1]{fontenc}
\usepackage{amssymb}
\usepackage{multicol}
\usepackage[active]{srcltx}
\usepackage{algorithm}
 \usepackage{mathtools}
\usepackage{array}
\usepackage[dvipsnames,svgnames]{xcolor}
\usepackage{tikz}
\usepackage{subfig}
\usepackage{graphicx}
\usepackage{xcolor}
\usepackage{ulem}
\usepackage{comment}
\providecommand{\U}[1]{\protect \rule{.1in}{.1in}}
\newtheorem{theorem}{Theorem}

\newtheorem{corollary}[theorem]{Corollary}

\newtheorem{definition}[theorem]{Definition}

\newtheorem{lemma}[theorem]{Lemma}

\newtheorem{proposition}[theorem]{Proposition}
\newtheorem{remark}[theorem]{Remark}

\begin{document}

\maketitle

\begin{abstract}
\noindent %{\bf Abstract}
We study differentiable strongly quasiconvex func\-tions for providing new properties for algorithmic and monotonicity purposes. Furthemore, we provide insights into the 
de\-crea\-sing behaviour of strongly quasiconvex functions, applying this for esta\-bli\-shing exponential convergence for first- and second-order 
gra\-dient systems without relying on the usual Lipschitz continuity 
assump\-tion on the gradient of the function. The explicit discretization of the 
first-order dynamical system leads to the gradient descent method while 
discretization of the second-order dynamical system with viscous damping 
recovers the heavy ball method. We establish the linear convergence of both 
methods under suitable conditions on the parameters as well as comparisons 
with other classes of nonconvex functions used in the gradient descent 
literature.

\medskip

\end{abstract}

%\centerline{\today}

\section{Introduction}

Given a differentiable function $h: \mathbb{R}^{n} \rightarrow \mathbb{R}$,
we are interested in solving the following minimization problem:
\begin{equation}
 \min_{x \in \mathbb{R}^{n}} \, h(x),\notag
\end{equation}
numerically via gradient methods, who simplest version takes the form:
Take $x^{0} \in \mathbb{R}^{n}$ and compute
\begin{equation}\label{gradient:for:intro}
 x^{k+1} = \, x^{k} - \beta_{k} \nabla h(x^{k}), ~ \forall ~ k \geq 0,
\end{equation}
where $\{\beta_{k}\}_{k}$ is a sequence of positive stepsizes to be
chosen.

The gradient method and its accelerated versions are fundamental and outstanding methods for finding local minimum points of differentiable functions. These methods are particularly valuable for classes of functions where every local minimum is also a global minimum, such as in (strongly) convex functions. The convergence properties of gradient-type methods for solving (strongly) convex functions are well understood \cite{Nesterov-book, P2} (see also \cite{Ant, ACPR, BS} among others).

In order to extend the good convergence rate properties of the gradient-type
methods for classes of nonconvex functions, a natural step is to consider
classes of generalized convex functions as, for instance, classes of
quasiconvex functions having the same useful property that every local
minimum is global minimum (e.g., pseudoconvexity, strong quasiconvexity,
etc). A crucial step in achieving this is to establish a decreasing property based on the gradient of these generalized convex functions, similar to the case of strongly convex functions.

In the case of differentiable quasiconvex functions, the cornerstone result
is due to Arrow and Enthoven in \cite{AE} (see also \cite[Theorem 
3.11]{ADSZ}), in which they characterized differentiable quasiconvex 
functions via a decreasing property for all points belonging to a specific 
sublevel set of the function. As a consequence of this result, a differentiable 
function is quasiconvex (resp. pseudoconvex) if and only if its gradient is 
quasimonotone (resp. pseudomonotone), see \cite{CM-Book,HKS} for 
instance. However, these decreasing properties alone are insufficient to guarantee exponential convergence to a local solution in continuous gradient descent, or a linear convergence rate in gradient-type methods. This limitation is highlighted in results like \cite[Theorem 4.1]{Bolte}, \cite[Theorems 1 and 2]{GOM}, and \cite[Theorem 3.3]{PK}, where convergence to critical points was shown under a quasiconvexity assumption. Thus, a stronger form of quasiconvexity must be assumed.

A promising and natural extension of strongly convex functions that could ensure good convergence rates for gradient methods is the class of strongly quasiconvex functions, introduced by Polyak in \cite{P}. Members of this class possess the valuable property that every local minimum is also a global minimum. It is worth noting that every strongly convex function is strongly quasiconvex, although the reverse is not true in general. Despite this, no gradient methods had been developed for this class of nonconvex functions for a long time, probably because the existence of solutions to the problem of minimizing strongly quasiconvex functions was uncertain until the recent contribution \cite{Lara-9}, which provided a fresh perspective on this class of functions.

Therefore, a natural subsequent step is to determine whether gradient-type methods can ensure convergence to the unique solution of the minimization problem for a differentiable strongly quasiconvex function. If so, the next objective is to establish the convergence rate. This precisely defines the motivation for our research.\\

As a first contribution, we study 
di\-ffe\-ren\-tia\-ble quasiconvex and strongly quasiconvex functions via the 
behaviour of their gradients. Furthermore, these results
provide us a de\-crea\-sing pro\-perty which is almost as good as the one for strongly convex functions and a\-llow us to provide relationships between 
differentiable strongly quasiconvex functions with other classes of nonconvex 
functions used in gradient type me\-thods such as quasi-strongly convex \cite{NP} 
and functions for which the Polyak-{\L}ojasiewicz property holds 
(see \cite{L,P1}). \\

As a second contribution,  we leverage the decreasing property to establish exponential convergence for both continuous first- and second- order dynamical systems, without requiring Lipschitz continuity of the gradient. Finally, by considering the explicit discretization of these dynamical systems, we obtain the gradient descent method and the Polyak's heavy-ball method \cite{P2}, respectively for minimizing strongly quasiconvex functions. We establish the linear convergence
rate for both methods under suitable conditions on the parameters.\\

The structure of the paper is as follows: In Section \ref{sec:2}, we present
preliminaries and basic definitions regarding generalized convexity and nonsmooth analysis. In Section \ref{sec:3}, we study differentiable strongly quasiconvex functions via
the behaviour of their gradients and, as a consequence, a new generalized
monotonicity notion is proposed. In Section \ref{sec:4}, we provide exponential
convergence for the first order continuous steepest descent dynamical system
and linear convergence rate for the gradient method under a locally Lipschitz
assumption on the gradient of the function. Finally, in Section \ref{sec:5}, we
establish exponential convergence for the second order continuous dynamical
system by assuming a weaker assumption than Lipschitz continuity on the
gradient of the function and, furthermore, we show the linear convergence
rate for the corresponding Heavy ball method.

\section{Preliminaries}\label{sec:2}

The inner product in $\mathbb{R}^{n}$ and the Euclidean norm are
denoted by $\langle \cdot,\cdot \rangle$ and $\lVert \cdot \rVert$,
respectively. The set $]0, + \infty[$ is denoted by $\mathbb{R}_{++}$.
Given any $x, y, z \in \mathbb{R}^{n}$ and any $\beta \in \mathbb{R}$,
the following relations hold:
\begin{align}
  \lVert \beta x + (1-\beta) y \rVert^{2} &= \beta \lVert x \rVert^{2} +
  (1 - \beta) \lVert y\rVert^{2} - \beta(1 - \beta) \lVert x - y \rVert^{2},
 \label{iden:1}\\
 \langle x - z, y - x \rangle &= \frac{1}{2} \lVert z - y \rVert^{2} -
 \frac{1}{2} \lVert x - z \rVert^{2} - \frac{1}{2} \lVert y - x \rVert^{2}.
 \label{3:points}
\end{align}

Given any extended-valued function $h: \mathbb{R}^{n} \rightarrow
\overline{\mathbb{R}}$, the effective domain of $h$ is defined by
${\rm dom}\,h := \{x \in \mathbb{R}^{n}: h(x) < + \infty \}$. It is said
that $h$ is proper if ${\rm dom}\,h$ is nonempty and $h(x) > - \infty$ for
all $x \in \mathbb{R}^{n}$. The notion of properness is important when
dealing with minimization pro\-blems.

It is indicated by ${\rm epi}\,h := \{(x,t) \in \mathbb{R}^{n} \times
\mathbb{R}: h(x) \leq t\}$ the epigraph of $h$, by $S_{\lambda} (h) :=
\{x \in \mathbb{R}^{n}: h(x) \leq \lambda\}$ the sublevel set of $h$ at
the height $\lambda \in \mathbb{R}$ and by ${\rm argmin}_{
\mathbb{R}^{n}} h$ the set of all minimal points of $h$. A function
$h$ is lower se\-mi\-continuous at $\overline{x} \in \mathbb{R}^{n}$
if for any sequence $\{x_k\}_{k} \in \mathbb{R}^{n}$ with $x_k
\rightarrow \overline{x}$, $h(\overline{x}) \leq \liminf_{k \rightarrow +
\infty} h(x_k)$. Furthermore, the current convention $\sup_{\emptyset}
h := - \infty$ and $\inf_{\emptyset} h := + \infty$ is adopted.

A function $h$ with convex domain is said to be
\begin{itemize}
 \item[$(a)$] convex if, given any $x, y \in \mathrm{dom}\,h$, then
 \begin{equation}\label{def:convex}
  h(\lambda x + (1-\lambda)y) \leq \lambda h(x) + (1 - \lambda) h(y),
  ~ \forall ~ \lambda \in [0, 1],
 \end{equation}

 \item[$(b)$] strongly convex on ${\rm dom}\,h$ with modulus $\gamma
 \in \, ]0, + \infty[$ if for all $x, y \in \mathrm{dom}\,h$ and all
 $\lambda \in[0, 1]$, we have
 \begin{equation}\label{strong:convex}
  h(\lambda y + (1-\lambda)x) \leq \lambda h(y) + (1-\lambda) h(x) -
  \lambda (1 - \lambda) \frac{\gamma}{2} \lVert x - y \rVert^{2},
 \end{equation}

 \item[$(c)$] quasiconvex if, given any $x, y \in \mathrm{dom}\,h$, then
 \begin{equation}\label{def:qcx}
  h(\lambda x + (1-\lambda) y) \leq \max \{h(x), h(y)\}, ~ \forall ~
  \lambda \in [0, 1],
 \end{equation}

 \item[$(d)$] strongly quasiconvex on ${\rm dom}\,h$ with modulus
 $\gamma \in \, ]0, + \infty[$ if for all $x, y \in \mathrm{dom}\,h$
 and all $\lambda \in[0, 1]$, we have
 \begin{equation}\label{strong:quasiconvex}
  h(\lambda y + (1-\lambda)x) \leq \max \{h(y), h(x)\} - \lambda(1 -
  \lambda) \frac{\gamma}{2} \lVert x - y \rVert^{2}.
 \end{equation}
 \noindent It is said that $h$ is strictly convex (resp. strictly
 quasiconvex) if the inequa\-li\-ty in \eqref{def:convex} (resp.
 \eqref{def:qcx}) is strict whenever $x \neq y$.
 \end{itemize}
 The relationship between all these notions is summarized below (we
 denote quasiconvex by qcx):
 \begin{align}\label{scheme}
  \begin{array}{ccccccc}
  {\rm strongly ~ convex} & \Longrightarrow & {\rm strictly ~ convex} &
  \Longrightarrow & {\rm convex}  \notag \\
  \Downarrow & \, & \Downarrow & \, & \Downarrow  \notag \\
  {\rm strongly ~ qcx} & \Longrightarrow & {\rm strictly ~ qcx} &
  \Longrightarrow & {\rm qcx}
  \end{array}
 \end{align}
All the reverse statements do not hold in general (see 
\cite{ADSZ,CM-Book,HKS,Lara-9}).

%For instance, the function $h_{1}(x) = \sqrt{\lVert x \rVert}$ is strongly 
%quasiconvex on any bounded convex set (see \cite[Theo\-rem 17]{Lara-9}) 
%without being convex  and the function $h_{2} (x) = \frac{x}{1 + \lvert x 
%\rvert}$ is strictly quasiconvex without being strongly quasiconvex on 
%$\mathbb{R}$ while the other counter examples are well-known 

Before continuing, let us show some examples of strongly quasiconvex 
functions which are not convex.

\begin{remark}\label{rem:exam}
 \begin{itemize}
  \item[$(i)$] Let $h: \mathbb{R}^{n} \rightarrow \mathbb{R}$ be given 
   by $h(x) = \sqrt{\lVert x \rVert}$. Clearly, $h$ is nonconvex, but it is 
   strongly quasiconvex on any $\mathbb{B} (0, r)$, $r > 0$, with modulus 
   $\gamma = \frac{1}{5^{\frac{1}{4}} 2^{\frac{5}{4}} r^{\frac{1}{2}}}$ 
   by \cite[Theorem 17]{Lara-9}.

  \item[$(ii)$] Let $A, B \in \mathbb{R}^{n\times n}$ be two symmetric 
  matrices, $a, b \in \mathbb{R}^{n}$, $\alpha, \beta \in \mathbb{R}$, and 
  $h: \mathbb{R}^{n} \rightarrow \mathbb{R}$ be the function given by:
  \begin{equation}
   h(x) = \frac{f(x)}{g(x)} := \frac{\frac{1}{2} \langle Ax, x \rangle + \langle 
   a, x \rangle + \alpha}{\frac{1}{2} \langle Bx, x \rangle + \langle b, x 
   \rangle + \beta}.
  \end{equation}
  Take $0 < m < M$ and define:
  $$K := \{x \in \mathbb{R}^{n}: ~ m \leq g(x) \leq M\}.$$ 
   If $A$ is a positive definite matrix and at least one of the following 
   conditions holds:
  \begin{enumerate}
   \item[$(a)$] $B = 0$ (the null matrix),
 
   \item[$(b)$] $f$ is nonnegative on $K$ and $B$ is negative semidefinite,
 
   \item[$(c)$] $f$ is nonpositive on $K$ and $B$ is positive semidefinite,
 \end{enumerate}
  then $h$ is strongly quasiconvex on $K$ with modulus $\gamma =
  \frac{\lambda_{\min} (A)}{M}$ by \cite[Corollary 4.1]{ILMY}, where 
  $\lambda_{\min} (A)$ is the minimum eigenvalue of $A$.

  \item[$(iii)$] Let $h_{1}, h_{2}: \mathbb{R}^{n} \rightarrow \mathbb{R}$ 
  be two strongly quasiconvex functions with modulus $\gamma_{1}, 
  \gamma_{2} > 0$, respectively. Then $h := \max\{h_{1}, h_{2}\}$ is 
  strongly quasiconvex with modulus $\gamma := \min\{\gamma_{1}, 
  \gamma_{2}\} > 0$ (straightforward).

  \item[$(iv)$] Let $\alpha > 0$ and $h: \mathbb{R}^{n} \rightarrow \mathbb{R}$ 
  be a strongly quasiconvex functions with modulus $\gamma > 0$. Then $\alpha h$ is 
  strongly quasiconvex with modulus $\gamma \alpha > 0$ (straightforward).
 \end{itemize}
\end{remark}

Recently, it was proved in \cite{Lara-9} that any lsc strongly
quasiconvex function has an unique minimizer as we recall next:

\begin{lemma}\label{exist:unique} {\rm (\cite[Corollary 3]{Lara-9})}
 Let $K \subseteq \mathbb{R}^{n}$ be a closed and convex set and $h:
 \mathbb{R}^{n} \rightarrow \overline{\mathbb{R}}$ a proper, lsc,
 and strongly qua\-si\-con\-vex function on $K \subseteq {\rm dom}\,h$
 with modulus $\gamma> 0$. Then, ${\rm argmin}_{K} h$ is a singleton.
\end{lemma}

The literature on strongly quasiconvex functions has been increasing very fast in the last years and several kind of algorithms have been developed for them (see \cite{GLM-1,ILMY,KL-2,Lara-9,LM}).

A function $h: \mathbb{R}^{n} \rightarrow \mathbb{R}$ its said to be
$L$-smooth on $K \subseteq \mathbb{R}^{n}$,  with $L> 0$, if it is differentiable on
$K$ and
\begin{equation}\label{L:smooth}
 \lVert \nabla h(x) - \nabla h(y) \rVert \leq L \lVert x - y \rVert, ~
 \forall ~ x, y \in K.
\end{equation}

For $L$-smooth functions, a fundamental result is the descent lemma
\cite[Lemma 1.2.3]{Nesterov-book}, that is, if $h$ is an $L$-smooth
function on a convex set $K$ with value $L  > 0$, then for every
$x, y \in K$, we have
\begin{equation}\label{descent:lemma}
 \lvert h(y) - h(x) - \langle \nabla h(x), y - x \rangle \rvert \leq
 \frac{L}{2} \lVert x - y \rVert^{2}.
\end{equation}

Another relevant result is the sufficient decrease lemma, which is given
as follows: Suppose that $h\in C_{L}^{1,1}(\mathbb{R}^n)$ (continuously
differentiable function with $L$-Lipschitz gradient, $L > 0$). Then for any
$x \in \mathbb{R}^n$ and $\beta >0$, holds
 \begin{align}\label{eq:dec}
  h(x) - h(x - \beta \nabla\, h(x)) \geq \beta \left(1 - \frac{\beta L}{2}
  \right) \| \nabla\,h(x) \|^2.
 \end{align}

Given a nonempty set $C$ in $\mathbb{R}^{n}$ and a single-valued
operator $T: \mathbb{R}^{n} \rightarrow \mathbb{R}^{n}$, $T$ is said
to be:
\begin{itemize}
 \item[$(a)$] monotone on $C$, if for all $x, y \in C$, we have
 \begin{equation}\label{def:monotone}
  \langle T(y) - T(x), y - x \rangle \geq 0.
 \end{equation}

 \item[$(b)$] strongly monotone on $C$ with modulus $\gamma > 0$,
  if for all $x, y \in C$, we have
 \begin{equation}\label{def:stronglymon}
  \langle T(y) - T(x), y - x \rangle \geq \gamma \lVert y - x \rVert^{2}.
 \end{equation}

 \item[$(c)$] pseudomonotone on $C$, if for all $x,y\in C$, we have
 \begin{equation}\label{def:pseudomonotone}
  \langle T(y), x - y \rangle \geq 0 ~ \Longrightarrow ~ \langle T(x),
  y - x \rangle \leq 0.
 \end{equation}

 \item[$(d)$] strongly pseudomonotone on $C$ with modulus
 $\gamma > 0$, if for all $x, y \in C$, we have
 \begin{equation}\label{def:strongly:pseudomon}
  \langle T(y), x - y \rangle \geq 0 ~ \Longrightarrow ~ \langle T(x),
  y - x \rangle \leq - \gamma \lVert y - x \rVert^{2}.
 \end{equation}

 \item[$(e)$] quasimonotone on $C$, if for all $x,y\in C$, we have
 \begin{equation}\label{def:quasimonotone}
  \langle T(y), x - y \rangle > 0 ~ \Longrightarrow ~ \langle T(x), y -
  x \rangle \leq 0.
 \end{equation}

 \item[$(f)$] strongly quasimonotone on $C$ with modulus $\gamma
  > 0$, if for all $x, y \in C$, we have
 \begin{equation}\label{def:strongly:quasimon}
  \langle T(y), x - y \rangle > 0 ~ \Longrightarrow ~ \langle T(x), y - x
  \rangle \leq - \gamma \lVert y - x \rVert^{2}.
 \end{equation}
\end{itemize}
The relationship between all these notions is summarized below (we
 denote monotone by mon):
 \begin{align*}
  \begin{array}{ccccccc}
  {\rm strongly ~ monotone} & \Longrightarrow & {\rm strongly ~
pseudomonotone} & \Longrightarrow & {\rm strongly ~ quasimonotone}
\notag \\
  \Downarrow & \, & \Downarrow & \, & \Downarrow & \, & \, \notag \\
  {\rm monotone} & \Longrightarrow & {\rm pseudomonotone} &
  \Longrightarrow & {\rm quasimonotone} & \, & \, \notag
  \end{array}
 \end{align*}
while the reverse statements do not hold in general (see
\cite{CM-Book,HKS}).

For other results on smooth and nonsmooth analysis, generalized monotonicity, generalized
convexity and strong quasiconvexity we refer to
\cite{ADSZ,CM-Book,GLM-1,JO3,HKS,KL-2,Lara-9,rock-1980,P,VNC-2} and references therein.

\section{Differentiable Strongly Quasiconvex Functions}\label{sec:3}

We begin this section with an  important result, which characterizes
differentiable strongly quasiconvex functions. This result is mentioned in \cite[Theorem 1]{JO3} as a consequence of more general results from \cite{VNC-2}.

\begin{theorem}\label{char:gradient} {\rm (\cite[Theorems 2 and 6]{VNC-2})}
 Let $K \subseteq \mathbb{R}^{n}$ be a convex set and $h: K
 \rightarrow \mathbb{R}$ a differentiable function. Then $h$ is strongly
 quasiconvex with modulus $\gamma \geq 0$ if and only if for every
 $x, y \in K$, we have
 \begin{equation}\label{gen:char}
  h(x) \leq h(y) ~ \Longrightarrow ~ \langle \nabla h(y), x - y \rangle
  \leq -\frac{\gamma}{2} \lVert y - x \rVert^{2}.
 \end{equation}
\end{theorem}

From Theorem \ref{char:gradient} we recover the characterization given by Arrow and Enthoven in \cite{AE} ($\gamma = 0$).

\begin{corollary}\label{coro:AE}
 Let $K \subseteq \mathbb{R}^{n}$ be a convex set and $h: K
 \rightarrow \mathbb{R}$ a differentiable func\-tion. Then $h$ is
 quasiconvex if and only if for every $x, y \in K$, we have
 \begin{equation}\label{char:AE}
  h(x) \leq h(y) ~ \Longrightarrow ~ \langle \nabla h(y), x - y \rangle
 \leq 0.
 \end{equation}
\end{corollary}

A straightforward sufficient condition for a differentiable strongly
quasiconvex function to be strongly convex is given below.

\begin{corollary}
 Let $K \subseteq \mathbb{R}^{n}$ be a convex set and $h: K \rightarrow
 \mathbb{R}$ a differentiable and strongly quasiconvex function with
 modulus $\gamma \geq 0$. If for every $x, y \in K$, we have
 \begin{equation}
  \langle \nabla h(x), x - y \rangle \geq 0, ~ \forall ~ x \in S_{h(y)} (h),
  \notag
 \end{equation}
 then $h$ is strongly convex with modulus $\gamma \geq 0$.
\end{corollary}

%It follows from Theorem \ref{char:gradient} that a di\-ffe\-rentiable
%function is strongly quasiconvex func\-tion if and only if its gradient
%belongs to the strong sublevel sub\-di\-ffe\-ren\-tial \eqref{eq:SS}
%at every point.

%\begin{corollary}\label{coro:ss-nonempty}
% Let $K \subseteq \mathbb{R}^{n}$ be a convex set and $h: K \rightarrow
% \mathbb{R}$ a differentiable function. Then $h$ is strongly quasiconvex
% with modulus $\gamma \geq 0$ if and only if  $\nabla h(x) \in \partial_{1,
% \gamma} h(x)$ for all $x \in {\rm dom}\,h \cap K$.
%\end{corollary}

%\begin{proof}
% By Theorem \ref{char:gradient}, $h$ is strongly quasiconvex with
% modulus $\gamma \geq 0$ if and only if relation \eqref{gen:char} holds
% if and only if $\nabla h(x) \in \partial_{1, \gamma} h(x)$.
%\end{proof}

Theorem \ref{char:gradient} provided a better understanding of the 
class of strongly quasiconvex functions since we can connect this class 
with other well-known classes of nonconvex functions which are extremely 
useful for algorithmic purposes.

A differentiable function $h: \mathbb{R}^{n} \rightarrow \mathbb{R}$ is 
said to be $\mu$-quasi-strongly convex, $\mu>0$, if (see \cite{NP})
\begin{equation}\label{eq:qsconvex}
 \langle \nabla h(x), x - x^* \rangle \geq h(x) - h(x^*) + \frac{\mu}{2}
 \|x - x^*\|^2, ~ \forall ~ x \in \mathbb{R}^{n}.
\end{equation}
where $x^*$ denotes the projection of $x$ onto ${\rm argmin}_{
\mathbb{R}^{n}}\,h$. It is worth to mention that $\mu$-quasi-strongly 
convex functions are not necessarily convex.

A sufficient condition for $\mu$-quasi-strongly convex functions for being
strongly quasiconvex is given below.

\begin{proposition}\label{N:P}
 Let $h: \mathbb{R}^{n} \rightarrow \mathbb{R}$ be a continuous 
 differentiable function. Assu\-ming that $h$ is $\mu$-quasi-strongly convex 
 function ($\mu>0$) and admits a unique minimizer, then $h$ is strongly 
 quasiconvex with modulus $\mu > 0$.
\end{proposition}

\begin{proof}
 Since $h$ has a unique minimizer, say $\overline{x}$, it follows from
 relation \eqref{eq:qsconvex} that
 \begin{align*}
 \langle \nabla h(x), x - \overline{x} \rangle & \geq h(x) - h(\overline{x}) +
 \frac{\mu}{2} \|x - \overline{x}\|^2, \\
 & \geq \frac{\mu}{2} \|x - \overline{x}\|^2, ~ \forall ~ x \in \mathbb{R}^{n}.
\end{align*}
Therefore, $h$ is strongly quasiconvex with modulus $\mu > 0$ by Theorem
\ref{char:gradient}.
\end{proof}

Following \cite{L,P1}, we also recall that a differentiable function $h:
\mathbb{R}^{n} \rightarrow \mathbb{R}$ satisfies the
Polyak-{\L}ojasiewicz (PL henceforth) property if there exists $\mu > 0$ such 
that
\begin{equation}\label{PL:def}
 \lVert \nabla h(x) \rVert^{2} \geq \mu (h(x) - h(\overline{x})), ~ \forall
 ~ x \in \mathbb{R}^{n},
\end{equation}
where $\overline{x} \in {\rm argmin}_{\mathbb{R}^{n}}\,h$. PL property is 
implied by strong convexity, but it allows for multiple minima and does 
not require any convexity assumption (see \cite{Ka-Nu-Sc} for instance).

Another useful consequence of Theorem \ref{char:gradient} is given below (see also \cite[Theorem 2]{K-1980}).

\begin{proposition}\label{prop:PL}
 Let $h: \mathbb{R}^{n}
 \rightarrow \mathbb{R}$ be a strongly quasiconvex function with modulus 
 $\gamma > 0$ and differentiable with $L$-Lipschitz continuous gradient 
 ($L>0$). Then the PL property holds with modulus $\mu := 
 \frac{\gamma^{2}}{2L} > 0$, that is,
 \begin{equation}\label{eq:PL}
  \lVert \nabla h(x) \rVert^{2} \geq \frac{\gamma^{2}}{2L} (h(x) -
  h(\overline{x})), ~ \forall ~ x \in  \mathbb{R}^n,
 \end{equation}
 where $\overline{x} = {\rm argmin}_{\mathbb{R}^n}\, h$.
 %(which is assumed to be nonempty).
\end{proposition}

\begin{proof}
 %Let $\overline{x} = {\rm argmin}_{K} h$. Then by
 Using Theorem \ref{char:gradient}
 and the Cauchy-Schwarz inequality, we have
 \begin{align*}
  & \frac{\gamma}{2} \lVert x - \overline{x} \rVert^{2} \leq \langle \nabla 
  h(x), x - \overline{x} \rangle \leq \lVert \nabla h(x) \rVert \lVert x - 
  \overline{x} \rVert \\
  & \hspace{1.0cm} \Longrightarrow \, \frac{\gamma}{2} \lVert x - 
  \overline{x} \rVert \leq \lVert \nabla h(x) \rVert, ~ \forall ~ x \in 
  \mathbb{R}^n.
 \end{align*}
 Therefore, by using relation \eqref{descent:lemma} and that
 $\nabla h(\overline{x}) = 0$, we have
 $$\lVert \nabla h(x) \rVert^{2} \geq \frac{\gamma}{2} \lVert x - \overline{x}
 \rVert^{2} \geq \frac{\gamma^{2}}{2 L} (h(x) - h(\overline{x})), ~ \forall ~
 x \in \mathbb{R}^n,$$
 which completes the proof.
\end{proof}

\begin{remark}
 \begin{enumerate}
  \item[$(i)$] The converse statement in Proposition \ref{prop:PL} does not 
  hold in general. Indeed, let us consider a system of nonlinear equations 
  as in \cite[Example 4.1.3]{Nesterov-book}:
  \begin{align} \label{Sist}
   g(x) = 0,
  \end{align} 
  with \( g = (g_1, \dots, g_m) : \mathbb{R}^n \to \mathbb{R}^m \) being a 
  differentiable vector function. Assuming that \( m \leq n \) and  \eqref{Sist} 
  admits a solution, say $x^{*}$, we consider the function
  $$h(x) = \frac{1}{2} \sum_{i=1}^{m} g_i^2(x).$$
  Let us assume in addition that $\mu :=\inf_{x\in K} \lambda_{\text{min}} 
  J(x) J^{T} (x) > 0$, where $K$ is some convex set containing $x^*$ and
  $$ J^T(x) = (\nabla g_1(x), \dots, \nabla g_m(x)).$$
  Note that
  $$\|\nabla h(x)\|^2 = \langle J(x)J^T(x) g(x) , g(x)\rangle \geq \mu 
  (g(x))^2 = \mu (h(x) - h(x^*)).$$
  Thus, $h$ satisfies the PL property, however, if $m<n$, then the set of 
  solutions is not necessarily a singleton, i.e., $h$ may not be strongly 
  quasiconvex by Lemma \ref{exist:unique}.

  \item[$(ii)$] The function $h: \mathbb{R} \rightarrow \mathbb{R}$ 
  given by $h(x)=x^{2}+3 \sin^{2}x$ is an example of a strongly 
  quasiconvex function satisfying the PL property and without being convex.
 \end{enumerate}
\end{remark}

In the following statement, we characterize differentiable strongly
quasiconvex functions via a generalized monotonicity notion for their
gradient. We emphasize that the generalized monotonicity property in
relation \eqref{new:mon} (see below) is new.

\begin{proposition}\label{char:genmon}
 Let $K \subseteq \mathbb{R}^{n}$ be a convex set and $h: K \rightarrow
 \mathbb{R}$ a differentiable function. If $h$ is strongly quasiconvex
 with modulus $\gamma \geq 0$, then for every $x, y \in K$, we have
 \begin{equation}\label{new:mon}
  \langle \nabla h(x), y - x \rangle > - \frac{\gamma}{2} \lVert y - x
  \rVert^{2} ~ \Longrightarrow ~ \langle \nabla h(y), x - y \rangle \leq -
  \frac{\gamma}{2} \lVert y - x \rVert^{2}.
 \end{equation}
\end{proposition}

\begin{proof}
Since $h$ is strongly quasiconvex with mo\-du\-lus
$\gamma \geq 0$, by Theorem \ref{char:gradient} we have for every
$x, y \in K$ that
$$h(x) \leq h(y) ~ \Longrightarrow ~ \langle \nabla h(y), x - y \rangle
  \leq -\frac{\gamma}{2} \lVert y - x \rVert^{2}.$$
By interchanging $x$ and $y$, we have
\begin{equation}\label{basic2}
 h(y) \leq h(x) ~ \Longrightarrow ~ \langle \nabla h(x), y - x \rangle
 \leq - \frac{\gamma}{2} \lVert y - x \rVert^{2}.
\end{equation}

Thus, if $\langle \nabla h(x), y - x \rangle > - \frac{\gamma}{2} \lVert
y - x \rVert^{2}$ with $x \neq y$, then $h(y) > h(x)$ by \eqref{basic2}.
Using \eqref{gen:char} we obtain $\langle \nabla h(y), x - y \rangle \leq
- \frac{\gamma}{2} \lVert y - x \rVert^{2}$, so \eqref{new:mon} holds.
\end{proof}

\begin{remark}\label{char:quasimon}
 The reverse statement in Proposition \ref{char:genmon} remains open, and we do not yet have a counter-example. Furthermore, if $\gamma = 0$, then Proposition \ref{char:genmon} reduces to the following characterization for differentiable quasiconvex functions: $h$ is quasiconvex if and only if $\nabla h$ is quasimonotone (see, for instance, \cite[Theorem 5.5.2$(b)$]{CM-Book}).
\end{remark}

If $K$ is open and $\gamma > 0$, then we can say something more.

\begin{proposition}\label{char:genmon2}
 Let $K \subseteq \mathbb{R}^{n}$ be an open convex and $h: K \rightarrow
 \mathbb{R}$ 
 a differentiable function. Consider the following statements:
 \begin{itemize}
  \item[$(a)$] $h$ is strongly quasiconvex with modulus $\gamma > 0$.

 \item[$(b)$] For each $x, y \in K$, the following implication holds:
\begin{equation}\label{ultrastrong2}
  \langle \nabla h(x), y - x \rangle \geq - \frac{\gamma}{2} \lVert y
  - x \rVert^{2} ~ \Longrightarrow ~ \langle \nabla h(y), x - y \rangle
  \leq - \frac{\gamma}{2} \lVert y - x \rVert^{2}.
 \end{equation}

 \item[$(c)$] For each $x, y \in K$, the following implication holds:
 \begin{equation}\label{new:mon2}
  \langle \nabla h(x), y - x \rangle > - \frac{\gamma}{2} \lVert y - x
  \rVert^{2} ~ \Longrightarrow ~ \langle \nabla h(y), x - y \rangle \leq -
  \frac{\gamma}{2} \lVert y - x \rVert^{2}.
 \end{equation}
 \end{itemize}
Then,
 $$(a) ~ \Longrightarrow ~ (b) ~ \Longrightarrow ~ (c).$$
\end{proposition}

\begin{proof}
 $(a) \Rightarrow (c)$ is Proposition \ref{char:genmon} while
 $(b) \Rightarrow (c)$ is straightforward.

 $(a) \Rightarrow (b)$: Since $(a) \Rightarrow (c)$, we only need to analyze the case when $\langle \nabla h(x), y - x \rangle
 = - \frac{\gamma}{2} \lVert y - x \rVert^{2}$. Since $K$ is open, take
 $t>0$ small enough such that $y_{t} := y + t(y-x) \in K$. Then,
\begin{align}
 \langle \nabla h(x), y_{t} - x \rangle & = (1 + t) \langle \nabla h(x),
 y - x \rangle = - (1 + t) \frac{\gamma}{2} \lVert y - x \rVert^{2} \notag \\
 & = - \frac{\gamma}{2 (1 + t)} \lVert y_{t} - x \rVert^{2} > - \frac{
 \gamma}{2} \lVert y_{t} - x \rVert^{2}. \label{for:rem}
\end{align}
Since $(a)$ holds, from Theorem \ref{char:gradient} we have $\langle
\nabla h(x), y_{t} - x \rangle > - \frac{\gamma}{2} \lVert y_{t} - x
\rVert^{2}$ implies $h(y_{t}) > h(x)$. Since $h$ is continuous, $h(y)
\geq h(x)$, and by using Theorem \ref{char:gradient} again, we obtain
\eqref{ultrastrong2}.
\end{proof}

\begin{remark}
 The reverse statement in Proposition \ref{char:genmon2} are open problems. We claim that all of them are equivalent.
\end{remark}

Before continuing, we recall the following definition.

\begin{definition}{\rm (\cite[equation (3)]{KK})}
 Let $K \subseteq \mathbb{R}^{n}$ be a convex set and $h: K \rightarrow
 \mathbb{R}$ a differentiable function. Then $h$ is said to be sharply
 quasiconvex with modulus $\gamma \geq 0$ if for every $x, y \in K$, the
 following implication holds:
 \begin{equation}\label{sharply:qcx}
  \langle \nabla h(y), x - y \rangle \geq 0 \Longrightarrow h(\lambda y
  + (1-\lambda)x ) \leq \max\{h(y), h(x)\} - \lambda (1-\lambda) \frac{
  \gamma}{2} \lVert y - x \rVert^{2}.
 \end{equation}
\end{definition}

\begin{remark}
 \begin{itemize}
  \item[$(i)$] Note that every strongly convex function is strongly
  quasiconvex, and every strongly quasiconvex function is sharply
  quasiconvex.

  \item[$(ii)$] If $\nabla h$ is strongly monotone with modulus $\gamma
  \geq 0$, then satisfies relation \eqref{new:mon}. Indeed, let $x, y \in K$
  such that $\langle \nabla h(y) - \nabla h(x), y - x \rangle \geq \gamma
  \lVert y - x \rVert^{2}$. Then,
  $$\langle \nabla h(y), y - x \rangle - \frac{\gamma}{2} \lVert y - x
  \rVert^{2} \geq \langle \nabla h(x), y - x \rangle + \frac{\gamma}{2}
  \lVert y - x \rVert^{2}.$$
 Hence, if $\langle \nabla h(x), y - x \rangle > - \frac{\gamma}{2}
 \lVert y - x \rVert^{2}$, then $\langle \nabla h(y), y - x \rangle >
 \frac{\gamma}{2} \lVert y - x \rVert^{2}$, i.e., $\nabla h$ satisfies
 \eqref{new:mon}.

 Furthermore, clearly, if $\nabla h$ satisfies \eqref{new:mon}, then
 $\nabla h$ is strongly pseudomonotone with modulus $\gamma \geq 0$.
 \end{itemize}
\end{remark}

In order to close the relationships between all these notions, we resume
the results below.

\begin{proposition}\label{gen:mon:rel}
 Let $K \subseteq \mathbb{R}^{n}$ be a convex set, $h: K \rightarrow
 \mathbb{R}$ be differentiable func\-tion and $\gamma \geq 0$. Then
  \begin{align}
  \begin{array}{ccccccc}
   h \, {\rm is} \, \gamma-{\rm strongly \, convex} & \Longrightarrow & h
   \, {\rm is} \, \gamma-{\rm strongly \, qcx} & \Longrightarrow & h \,
   {\rm is} \, \gamma-{\rm sharply}. \notag \\
   \Updownarrow & \, & \Downarrow & \, & \Updownarrow & \, \notag \\
   \nabla h \, {\rm is} \, \gamma-{\rm strongly \, mon} &
   \Longrightarrow & \nabla h \, {\rm satisfies} \,
   \eqref{new:mon} & \Longrightarrow & \nabla h \, {\rm is} \, \frac{
   \gamma}{2}-{\rm strongly \, pseudo} & \notag
  \end{array}
 \end{align}
\end{proposition}

\begin{proof}
 We simple note that $h$ is sharply quasiconvex with modulus
 $\gamma > 0$ if and only if $\nabla h$ is strongly pseudomonotone
 with modulus $\frac{\gamma}{2} > 0$ by \cite[Proposition 2.1]{KK}
 while the case when $\gamma = 0$ is straightforward.
\end{proof}

\section{First-Order Gradient Dynamics}\label{sec:4}

In this section, we study the minimization problem
\begin{align}\label{min.h}
 \min_{x\in \mathbb{R}^{n}} h (x).
\end{align}
where $h: \mathbb{R}^n \to \mathbb{R}$ is a continuously differentiable
function. Our goal is the study of problem \eqref{min.h} via a first order
dynamical system and its corresponding explicit discretization and, in virtue
of Proposition \ref{prop:PL}, without assuming a Lipschitz continuity property 
on the gradient of the function.

\subsection{Continuous Dynamics}

We consider the following dynamical system:
\begin{align}
\left\{
 \begin{array}{ll}\label{eq:Gdy}
  \dot{x}(t) + \nabla h (x(t)))= 0, \, t > 0, \\ [2mm]
  x(t_0) = x_0.
 \end{array}
 \right.
\end{align}

In the following theorem, we provide exponential convergence of the
 trajectories generated by \eqref{eq:Gdy} to the unique solution of
problem \eqref{min.h}.

\begin{theorem}\label{theo11}
 Let the function $h: \mathbb{R}^n \rightarrow \mathbb{R}$ be
 continuously differentiable and strongly quasiconvex with modulus
 $\gamma > 0$ and $\overline{x} = {\rm argmin}_{\mathbb{R}^{n}}\,h$.
 Then the fo\-llo\-wing assertions hold:
 \begin{itemize}
  \item[$(a)$] $t\mapsto h(x(t))$ is nonincreasing.

  \item[$(b)$] Any trajectory $x(t)$ to the dynamical system
  \eqref{eq:Gdy} satisfy that
  \begin{equation}\label{exp:conv1}
   \lVert x(t) - \overline{x} \rVert \leq \lVert x_0 - \overline{x} \rVert e^{-
   \frac{\gamma}{2} t},
  \end{equation}
  i.e., $x(t)$  converges exponentially to the unique solution of
  \eqref{min.h};

  \item[$(c)$]
  %If in addition the function $h$ is locally Lipschitz continuous, then
  For any trajectory $x(t)$ to the dynamical system \eqref{eq:Gdy}
  there exists $T > 0$ and $L > 0$ such that 
  \begin{equation}\label{descent:conse1}
   h(x(t)) - h(\overline{x}) \leq \min\left\{\frac{L}{2} \lVert x_0 - \overline{x}
   \rVert e^{- \frac{\gamma}{2} t}, \left( h(x_0) - h(\overline{x}) \right)
   e^{- \frac{\gamma^2}{2L} t} \right\}, ~ \forall ~ t \geq T,
  \end{equation} 
  as a consequence, $ h(x(t))$ converges exponentially to $h^{*} =
  \min_{\mathbb{R}^{n}} h$.
 \end{itemize}
\end{theorem}

\begin{proof}
 $(a)$: Note that for each $t>t_0$, we have
 \begin{align}\label{eq:PL1}
\frac{d}{dt}  h (x(t)) = \langle \nabla h(x(t)), \dot{x}(t) \rangle =
 -  \| \nabla h (x(t)) \|^2 \leq 0.
 \end{align}
 Hence, $h$  decreases along the solutions and does so strictly unless
 it hits a critical point.

 $(b)$: Let $\overline{x} = {\rm argmin}_{\mathbb{R}^{n}} h$. Then
 for every $t \in [0, + \infty[$, we define the auxiliar func\-tion $E(t): =
 \frac{1}{2} \| x(t) - \overline{x} \|^2$. Then
 \begin{align}
  \dot{E} (t) = \langle x(t) - \overline{x}, \dot{x}(t) \rangle = - \langle
  x(t) - \overline{x}, \nabla h(x(t)) \rangle.\notag
 \end{align}
 Since $\overline{x} = {\rm argmin}_{\mathbb{R}^{n}}h$ and $h$ is
 strongly quasiconvex with modulus $\gamma \geq 0$, it follows from
 Theorem \ref{char:gradient} that
 \begin{align}\label{eq:TQ}
  \langle \nabla h(x(t)), x(t) - \overline{x} \rangle \geq \frac{\gamma
  }{2} \lVert x(t) - \overline{x} \rVert^{2}.
 \end{align}
 Hence
\begin{equation}
 \dot{E}(t) \leq - \gamma E(t).\notag
 \end{equation}
By integrating, we obtain
\begin{equation}\label{eq:exp}
 E(t) \leq E(t_0) e^{-\gamma t} ~ \Longleftrightarrow ~ \lVert x(t) -
 \overline{x} \rVert \leq \lVert x_0 - \overline{x} \rVert e^{- \frac{
 \gamma}{2} t}.
\end{equation}

 $(c)$: We know that any continuously differentiable function is locally
 Lipschitz continuous. Hence, since $x(t)$ converges exponentially to the
 unique solution $\overline{x}(t)$, for some $T$ sufficiently large, the
 function $h$ is Lipschitz continuous around  $\overline{x}(t)$. Therefore,
 there exists $L>0$ such that
 \begin{equation}\label{cota:01}
  h(x(t)) - h(\overline{x}) \leq \frac{L}{2} \lVert x(t) - \overline{x}
  \rVert \leq \frac{L}{2} \lVert x_0 - \overline{x} \rVert e^{-\gamma t},
 \end{equation}
 for all $t \geq T$.

 Moreover, using \eqref{eq:PL1} and \eqref{PL:def}, we have
 $$\frac{d}{dt}\left(h(x(t))- h(\overline{x})\right)=\langle \nabla h(x(t)),
 \dot{x}(t)\rangle= -\lVert \nabla\, h(x(t))\rVert^2 \leq - \frac{\gamma^2
 }{2L}\left(h(x(t)) - h(\overline{x})\right).$$
 Integrating, we have
 \begin{equation}\label{cota:02}
 h(x(t))- h(\overline{x})\leq \left(h(x_0) - h(\overline{x})\right)
 e^{-\frac{\gamma^2}{2L}t}.
 \end{equation}
 Finally, relation \eqref{descent:conse1} follows from \eqref{cota:01} and
 \eqref{cota:02}. 
\end{proof}

\subsection{Steepest Descent Method}

 As a discretization version of \eqref{eq:Gdy}, let us consider the classical
 gradient descent method (see, for instance, \cite[page 68]{Nesterov-book}).

\begin{algorithm}[H]%\label{EMSQ:EP}
 \caption{The Gradient Method}\label{GM}
 \begin{description}
  \item[Step 0.] Take $x^{0} \in \mathbb{R}^{n}$, $k=0$ and a
  sequence $\{\beta_{k}\}_{k} \subseteq \mathbb{R}_{++}$.
  %bounded away from $0$.

  \item[Step 1.] Compute $\nabla h(x^{k})$ and
   \begin{align}
    x^{k+1} = x^{k} - \beta_{k} \nabla h(x^{k}). \label{step:gradient}
   \end{align}

   \item[Step 2.] If $x^{k+1} = x^{k}$, then STOP, $x^{k} \in {\rm
   argmin}_{\mathbb{R}^{n}}\,h$. Otherwise, take $k=k+1$ and go to
   Step 1.
  \end{description}
\end{algorithm}

Before continuing with the convergence analysis, we note the following.

\begin{remark}\label{rem:decrease}
 Let $h$ be a differentiable strongly quasiconvex function with modulus
 $\gamma > 0$. If $\overline{x} = {\rm argmin}_{\mathbb{R}^{n}}\,h$,
 then it follows from Theorem \ref{char:gradient} that for every $k \in
 \mathbb{N}$,
 \begin{align}
 \langle \nabla
  h(x^{k}), \overline{x} - x^{k} \rangle \leq - \frac{\gamma}{2} \lVert
 x^{k} - \overline{x} \rVert^{2}, \label{rel:min-k}
 \end{align}
  where $\{x^{k}\}_k$ is the sequence generate by Algorithm \ref{GM}.
\end{remark}

In order to study the convergence of Algorithm \ref{GM} under a locally
Lipschitz continuity property for the gradient of $h$, we note also the
following.

\begin{remark}\label{loc:lips}
 \begin{itemize}
  \item[(i)] Assuming that $h$ has a Lipschitz continuous gradient with
   modulus $L>0$ and $0< \beta_k \leq \frac{2}{L}$, it follows from \eqref{eq:dec} and \eqref{step:gradient} that
   \begin{align}\label{eq:TU}
     h(x^k) - h(x^{k+1}) \geq \beta_k\left(1-\frac{\beta_k L}{2}\right)
   \|\nabla\,h(x^{k})\|^2 \geq 0, ~ \forall ~ k \in \mathbb{N}.
   \end{align}
 \item[$(ii)$] Since $h$ is strongly quasiconvex, its level set are compact
 by \cite[Theorem 1]{Lara-9}, thus given any starting point $x^{0} \in K$,
 the set $S_{h(x^{0})} (h)$ is compact. Furthermore, if $\nabla h$ is
 assumed to be locally Lipschitz continuous, it is Lipschitz continuous on
 bounded sets, in particular, there exists $L_{0} > 0$ such that $\nabla h$
 is Lipschitz continuous on $S_{h(x^{0})} (h)$. Hence, if $\beta_{k} \leq
 \frac{2}{L_{0}}$ for all $k \in \mathbb{N}$, the sequence $\{x^{k}\}_{k}
 \subseteq S_{h(x^{0})} (h)$ by part $(i)$ of this Remark, thus
 $\{x^{k}\}_{k}$ is bounded too, and fo\-llo\-wing \cite{MM}, we
 know that $\nabla h$ is Lipschitz continuous on $\mathcal{C} =
 \overline{{\rm{conv}}}\{\overline{x},  x^0, x^1, \ldots\}$. Hence,
 $$\|\nabla h(x^{k+1}) -\nabla h(x^{k})\| \leq L_{0} \|x^{k+1} - x^{k}\|,
 ~ \forall ~ k \in \mathbb{N}.$$
 Therefore, for Algorithm \ref{GM}, we may assume that $h$ has a locally
 Lipschitz continuous gradient.
 \end{itemize}
\end{remark}

By using the previous remark, we obtain the following.

\begin{proposition}\label{linearGraddescent}
 Let $h: \mathbb{R}^{n} \rightarrow \mathbb{R}$ be a strongly
 quasiconvex function with mo\-du\-lus $\gamma > 0$ and differentiable
 with locally Lipschitz continuous gradient, $\{\beta_{k}\}_k$ be a
 positive sequence  such that $0 < \beta_{k} \leq  \min \left\{
 \frac{\gamma}{L_0^2}, \frac{2}{L_{0}}\right\}$, then for every
 $k \in \mathbb{N}$, we have
 \begin{equation}\label{linearIne}
  \|x^{k+1} - \overline{x}\|^2 \leq  \left (1-\beta_{k} (\gamma -
  \beta_{k} L_0^2) \right )\|x^{k} - \overline{x}\|^2,
 \end{equation}
 where $\{x^{k}\}_k$ is the sequence generate by Algorithm \ref{GM}.
\end{proposition}

\begin{proof}
 From \eqref{step:gradient}, \eqref{gen:char} and the fact that	
 $\nabla   h(\overline{x}) = 0$, we have
	\begin{eqnarray*}
		\|x^{k+1} - \overline{x}\|^2
		&=& 	\|x^{k} - \beta_{k} \nabla h(x^k)-  \overline{x}\|^2 \\
		&=& 	\|x^{k} - \overline{x}\|^2  -2 \beta_{k} \left< \nabla h(x^k), x^k -\overline{x}\right> + \beta_{k}^2 \|\nabla h(x^k)\|^2\\
		&\leq& 	\|x^{k} - \overline{x}\|^2  - \gamma \beta_{k} 	\|x^{k} - \overline{x}\|^2  +
		\beta_{k}^2 \|\nabla h(x^k) - \nabla  h(\overline{x})\|^2\\
		&\leq& 	\|x^{k} - \overline{x}\|^2  - \gamma \beta_{k} 	\|x^{k} - \overline{x}\|^2  +
		\beta_{k}^2 L_0^2  \|x^{k} - \overline{x}\|^2 \\
		&=& \left (1-\beta_{k} (\gamma - \beta_{k} L_0^2) \right )\|x^{k} - \overline{x}\|^2,
	\end{eqnarray*}
 and the result follows.
\end{proof}

As a consequence, we obtain the convergence result.

\begin{theorem}\label{LC:2GD}
	Let $h: \mathbb{R}^{n} \rightarrow \mathbb{R}$ be a strongly
	quasiconvex function with mo\-du\-lus $\gamma > 0$ and differentiable
	with locally Lipschitz continuous gradient. Let $\overline{x} = {\rm argmin}_{
		\mathbb{R}^{n}}\,h$ and  $\{\beta_{k}\}_{k}$ be a
	positive sequence satisfying
	\begin{equation} \label{stepsize}
	0< \underline{\beta}  \leq \beta_{k} \leq \overline{\beta} < \min \left\{ \frac{\gamma}{L_0^2}, \frac{2}{L_{0}}\right\}.
	\end{equation}
  Then the sequence $\{x^{k}\}_{k}$, generated by Algorithm \ref{GM}, converges linearly to the unique solution  $\overline{x}$.
\end{theorem}
\begin{proof}
	From \eqref{linearIne} and \eqref{stepsize} we have
	\begin{eqnarray*}
		\|x^{k+1} - \overline{x}\|^2
		&\le & \left (1-\beta_{k} (\gamma - \beta_{k} L_0^2) \right )\|x^{k} - \overline{x}\|^2 \\
 & \le & \left (1- \underline{\beta} (\gamma - \overline{\beta} L_0^2) \right)
 \|x^{k} - \overline{x}\|^2,
\end{eqnarray*}
 which implies that $\{x^{k}\}_{k}$  converges linearly to the unique
 solution  $\overline{x}$ with a linear rate of at least $q = \sqrt{1 -
 \underline{\beta} (\gamma - \overline{\beta} L_0^2)} \in \, ]0, 1[$.
\end{proof}

\begin{remark}
 Note that if the function $h$  is strongly quasiconvex  with mo\-du\-lus
 $\gamma > 0$ then it is strongly quasiconvex  with any mo\-du\-lus
 smaller than $\gamma$. Hence if  $h: \mathbb{R}^{n} \rightarrow
 \mathbb{R}$ is a strongly quasiconvex function with mo\-du\-lus $\gamma
 > 0$ and di\-ffe\-ren\-tia\-ble with locally Lipschitz continuous gradient, we
 can always assume that $\gamma \leq  2L_0$ and
 $\beta_{k} < \frac{\gamma}{L_0^2} \leq \frac{2}{L_0}$.
 Now if we consider $q(\beta_k) = (1-\beta_k(\gamma-\beta_k L_0^2))$,
 then the minimal value of $q$ is $q^* = 1-\frac{\gamma^2}{4L_0^2}$
 when $\beta_k = \beta^* =  \frac{\gamma}{2L_0^2}$ for all $k$.
\end{remark}

\begin{corollary}\label{fv:estimates}
 Assume that hypotheses of Theorem fulfilled of Theorem \ref{LC:2GD} holds, $\gamma <2L_0$ and
 $\beta_{k} < \frac{\gamma}{L_0^2} \leq \frac{2}{L_0}$, Then we have an
 optimal convergence rate for the functional values:
 \begin{align}\label{fv:01}
  h(x^k) - h(\overline{x})\leq \left( 1- \frac{\gamma^2}{4L_0^2} \right)^{k-1}
  \|x^0 - \overline{x}\|^2,
 \end{align}
 and
 \begin{align}\label{fv:02}
  h(x^k) - h(\overline{x})\leq \left( 1- \frac{\gamma^3}{4L_0^3} \left(1 -
  \frac{\gamma} {4L_0}\right) \right)^{k-1} (h(x^0) - h(\overline{x}))
 \end{align}
 where $\{x\}_k$ is the sequence generated by Algorithm \ref{GM},
\end{corollary}

\begin{proof}
 Relation \eqref{fv:01} follows directly from \eqref{descent:lemma} and
 \eqref{linearIne}. To prove \eqref{fv:02}, it follows from \eqref{eq:PL} and
 \eqref{eq:TU}  that
 \begin{align}
  h(x^{k+1}) -  h(x^k) &\leq -\beta_k\left(1-\frac{\beta_k L}{2}\right)
  \|\nabla\,h(x^{k})\|^2 \notag \\
  & \leq -\frac{\gamma^3}{4L_0^3}\left(1-\frac{\gamma}{4L_0}\right)
  \left( h(x^{k}) - h(\overline{x}) \right), \notag
 \end{align}
 which is equivalent to
 \begin{align}
  h(x^{k+1}) -  h(\overline{x}) - (h(x^{k}) -  h(\overline{x})) & \leq -
  \frac{\gamma^3}{4L_0^3}\left(1-\frac{\gamma}{4L_0}\right) \left(
  h(x^{k}) - h(\overline{x}) \right). \notag
 \end{align}
This implies \eqref{fv:02}.
\end{proof}

\section{Second-Order Gradient Dynamic and Discretization}\label{sec:5}

As in the previous section, we will study problem \eqref{min.h} via a second 
order dynamical system without assuming a Lipschitz continuity property 
on the gradient of the function.

\subsection{Continuous dynamical system}
We consider the following gradient dynamical system:
\begin{align}
\left\{
 \begin{array}{ll}\label{eq:2Gdy}
  \ddot{x}(t) + \alpha\dot{x}(t) + \nabla h (x(t))= 0, \, t > 0, \\
 [2mm]
  x(0) = x_0, \quad \dot{x}(0) = v_0.
 \end{array}
 \right.
\end{align}

For the second-order continuous gradient dynamic \eqref{eq:2Gdy}, we
provide exponential convergence to a solution of problem \eqref{min.h} when $h$ is strongly quasiconvex.
We emphasize that in contrast to requiring Lipschitz continuity assumptions
on the gradient of $h$, we use a weaker hypothesis: there exists
$\kappa \in \, ]0, +\infty[$ such that for every  trajectory $x(t)$ of the
dynamical system \eqref{eq:2Gdy} we have

 \begin{equation}\label{weak:quasiconvex}
  \langle \nabla h(x(t)), x(t) - \overline{x} \rangle \geq \kappa (h(x(t)) -
  h(\overline{x})),
 \end{equation}
 where $\overline{x} = {\rm argmin}_{\mathbb{R}^n} h$. This assumption
 holds trivially with $\kappa =1$ for convex and strongly convex functions,
 hence can be considered as a generalized convexity assumption. Assumption
 \eqref {weak:quasiconvex} was used in \cite{CEG}, and later in
 \cite{AC2017,ADR2} with the stronger requirement that $\kappa \ge 1$.

Note that if the function $h: \mathbb{R}^n \rightarrow \mathbb{R}$ is
differentiable with $L$-Lipschitz gradient and strongly quasiconvex with
modulus $\gamma > 0$, then it follows from \eqref{gen:char},
\eqref{descent:lemma} and $\nabla h(\overline{x}) = 0$ that
$$ \langle \nabla h(x), x -  \overline{x} \rangle \geq \frac{\gamma}{2}
 \lVert x - \overline{x} \rVert^{2} \geq \frac{\gamma}{L} (h(x) -
 h(\overline{x})),$$
which implies \eqref{weak:quasiconvex} with $\kappa = \frac{\gamma}{L}$.

\begin{theorem}\label{theo12}
 Let $h: \mathbb{R}^n \rightarrow \mathbb{R}$ be a differentiable and 
 strongly quasiconvex function with modulus $\gamma > 0$ and $\overline{x} 
 = {\rm argmin}_{\mathbb{R}^n} h$. Suppose that assumption 
 \eqref{weak:quasiconvex} holds. Then any trajectory $x(t)$ generated by 
 \eqref{eq:2Gdy} converges exponentially to the unique solution $\overline{x}$ 
 of problem \eqref{min.h}, and the function values $h(x(t))$  converge exponentially to the optimal value $h(\overline{x})$.
\end{theorem}

\begin{proof}
 Let $h^* = \min_{\mathbb{R}^n}\, h$, we consider the following
 energy (Lyapunov)  function
 $$\Sigma(t): = h(x(t)) - h^* + \frac{1}{2} \|\lambda (x(t) -
 \overline{x}) + \dot{x}(t)\|^2 + \frac{\xi}{2} \|x(t)  - \overline{x}\|^2,$$
 where $\lambda$ and $\xi$ are some positive parameters to be determined
 later.  Differentiating the energy function $\Sigma$ we have
 $$\dot{\Sigma} (t) = \langle  \nabla h(x(t)), \dot{x} (t) \rangle +
 \langle \lambda (x(t) - \overline{x}) + \dot{x}(t), \lambda \dot{x}(t)
 + \ddot{x}(t)\rangle + \xi \langle x(t) - \overline{x}, \dot{x}(t)\rangle.$$
Combining with \eqref{eq:2Gdy} we deduce
$$ \dot{\Sigma} (t) = -\lambda \langle  \nabla h(x(t)), x (t)-\overline{x} 
\rangle + (\lambda - \alpha)\|\dot{x}(t)\|^2 + (\xi+\lambda (\lambda - 
\alpha)) \langle x(t) - \overline{x}, \dot{x}(t) \rangle. $$
On the other hand, from \eqref{gen:char} (see Theorem
\ref{char:gradient}) and \eqref{weak:quasiconvex} we can have
$$ - \lambda \langle  \nabla h(x(t)), x (t)-\overline{x} \rangle \leq -
\lambda \left (\frac{\gamma}{4} \|x(t) - \overline{x}\|^2 +
\frac{\kappa}{2} (h(x(t)) -h^*) \right).$$
Hence, by combining the last two inequalities, we have
\begin{eqnarray*}
\dot{\Sigma}(t) &\leq&  -\lambda \left (\frac{\gamma}{4} \|x(t) - \overline{x}\|^2+\frac{\kappa}{2} (h(x(t)) -h^*) \right) \\
&&\quad +  (\lambda - \alpha)\|\dot{x}(t)\|^2
+ (\xi+\lambda (\lambda-\alpha)) \langle x(t) - \overline{x}, \dot{x}(t)\rangle,
\end{eqnarray*}
which implies
\begin{eqnarray} \label{mainIne}
 \dot{\Sigma}(t)  + \frac{	\lambda \kappa}{2} \Sigma  (t) \leq
 \frac{\lambda}{4}\left(\kappa \lambda^2+ \kappa \xi -\gamma
 \right) \|x(t) - \overline{x}\|^2 +\left(\lambda +\frac{\lambda
 \kappa}{4} -\alpha \right) \|\dot{x}\|^2 \notag \\
 +\left(\xi+\lambda (\lambda - \alpha) +\frac{\lambda^2 \kappa}{2}
 \right) \langle  x(t) - \overline{x}, \dot{x}\rangle.
\end{eqnarray}
To establish the exponential convergence of $\Sigma(t)$ to $0$,
it remains to choose the positive parameters $\lambda$ and $\xi $ such
that the right-hand side of \eqref{mainIne} is non-positive.

For simplicity, let us choose $\xi = \lambda^2$ then we would require
$$ \kappa \lambda^2+ \kappa \xi -\gamma = 2 \kappa \lambda^2 - 
\gamma \le 0,$$
$$
\lambda +\frac{\lambda \kappa}{4} -\alpha \le 0,
$$
and
$$
\xi+\lambda (\lambda - \alpha) +\frac{\lambda^2 \kappa}{2} = \left(2+\frac{\kappa}{2} \right) \lambda^2 - \alpha \lambda \le 0.
$$
The above inequalities hold whenever $\lambda \leq 
\min\{ \sqrt{\frac{\gamma}{2\kappa}}, \frac{2\alpha}{\kappa+4}\}$.

With this choice of parameters, we can deduce from \eqref{mainIne} that
$\Sigma(t)$  converges exponentially to $0$. Indeed, from \eqref{mainIne}
%, \eqref{eq:VT} and \eqref{eq:VT1}, for every $t>0$,
we have
\begin{align}
  \dot{\Sigma}(t)+ \frac{\lambda\kappa}{2} \Sigma(t)\leq 0.\notag
 \end{align}
 Then applying Gronwall's inequality, we obtain
 \begin{equation}
  \Sigma(t) \leq \Sigma(0) \, e^{- \frac{\lambda\kappa}{2} t}\notag.
 \end{equation}
Since
$$
h(x(t)) - h(\overline{x}) + \frac{\xi}{2} \|x(t)  - \overline{x}\|^2 \le \Sigma(t),
$$
it is clear that  $x(t)$ converges exponentially to the unique solution
$\overline{x}$ and and the function values $h(x(t))$  converges
exponentially to the optimal value $h(\overline{x})$.
\end{proof}

\begin{remark}
 As a consequence of Theorem \ref{theo12}, the following asymptotic
 exponential convergence rate hold
 \begin{align}
  & h(x(t)) - h^{*} = \mathcal{O} \left(e^{-\frac{\lambda \kappa}{2} t}
  \right), \label{eq:inner.iter1} \\
  &  \|x(t) - \overline{x}\| = \mathcal{O} \left(e^{-\frac{\lambda \kappa}{4} t}  \right), \label{eq:inner.iter2}
 \end{align}
 with $\lambda = \min\{ \sqrt{\frac{\gamma}{2\kappa}}, \frac{2\alpha}{\kappa+4}\}$.
 Note that we can always choose $\alpha > 0$ such that
 $$\lambda = \min \left \{ \sqrt{\frac{\gamma}{2\kappa}}, \frac{2
 \alpha}{\kappa+4} \right \} = \sqrt{\frac{\gamma}{2\kappa}}$$
 and the convergence rate will be $\mathcal{O} \left(e^{-\sqrt{\gamma
 \kappa} t} \right)$  for both objective values and  trajectories. If
 $\kappa=1$ (e.g. strongly convex functions), then this rate
 $\mathcal{O} \left(e^{-\sqrt{\gamma} t} \right)$ is faster than the
 one in the first-order  dynamical system \eqref{eq:Gdy}  (which is
$\mathcal{O} \left(e^{-\gamma t} \right)$), whenever $\gamma \leq 1$.
\end{remark}

\subsection{Discrete system:  gradient descent with momentum}\label{sec:6}

Following \cite{ADR,Vu} let us consider an explicit discretization of
\eqref{eq:2Gdy}. We take a time step $\eta > 0$, and set
$t_k := k \eta$, and  $x_{k} := x(t_k)$. By a  finite-difference scheme
in \eqref{eq:2Gdy} with centered second-order variation, one obtains
\begin{align}
 & \hspace{1.0cm} \frac{x_{k+1} - 2 x_{k} + x_{k-1}}{\eta^{2}} +
 \frac{\alpha}{\eta} ( x_{k} - x_{k-1}) + \nabla h(x_{k}) = 0. \notag%\\
 %& \Longleftrightarrow \, x_{k+1} - 2 x_{k} + x_{k-1} + \alpha \eta ( x_{k+1} - x_{k-1}) + \eta^{2} \nabla h(x_{k}) = 0. \label{discre:2nd}
\end{align}
Rearranging the above relation, we get
$$x_{k+1} - x_{k}- (1- \alpha \eta)(x_{k} - x_{k-1})+\eta^2\nabla\,h(x^k).$$
Setting $\theta:=1- \alpha \eta$ and $\beta:= \eta^2$, we have the
following well known Heavy-ball algorithmic  scheme:
\begin{align}
 \begin{array}{ll}\label{inertial:alg}
  x_{k+1} = x_{k} + \theta (x_{k} - x_{k-1}) - \beta \nabla h(x_{k}).
 \end{array}
\end{align}

\begin{comment}
By explicit discretization to \eqref{eq:2Gdy}, we have
\begin{align}
 & \hspace{1.0cm} \frac{x_{k+1} - 2 x_{k} + x_{k-1}}{\eta^{2}} +
 \frac{\alpha}{\eta} ( x_{k+1} - x_{k-1}) + \nabla h(x_{k}) = 0 \notag \\
 & \Longleftrightarrow \, x_{k+1} - 2 x_{k} + x_{k-1} + \alpha \eta (
 x_{k+1} - x_{k-1}) + \eta^{2} \nabla h(x_{k}) = 0. \label{discre:2nd}
\end{align}
By simple algebra, we obtain the following:
\begin{equation}\label{main:eq:inertial}
 x_{k+1} = x_{k} + \left( \frac{1 - \alpha \eta}{1 + \alpha \eta} \right)
 (x_{k} - x_{k - 1}) - \frac{\eta^2}{1 + \alpha\eta } \nabla h(x_{k}).
\end{equation}
Then, we take
\begin{equation}\label{param:inertial-alg}
 \theta = \frac{1 - \alpha \eta}{1 + \alpha \eta}, ~~~~~
 \beta = \frac{\eta^2}{1 + \alpha \eta}.
\end{equation}
Hence, we have the following well known Heavy-ball algorithmic  scheme:
\begin{align}
 \begin{array}{ll}\label{inertial:alg}
  x_{k+1} = x_{k} + \theta (x_{k} - x_{k-1}) - \beta \nabla h(x_{k}).
 \end{array}
\end{align}
The linear convergence of the Heavy-ball algorithm is presented in the
following Theorem.
\end{comment}

The linear convergence of the Heavy-ball algorithm is presented in
the following statement.

\begin{theorem}\label{convergenceHeavyball}
 Let $h: \mathbb{R}^{n} \rightarrow \mathbb{R}$ be a differentiable with 
 $L$-Lipschitz gradient and strongly quasiconvex function with modulus 
 $\gamma > 0$, $\overline{x} = {\rm argmin}_{\mathbb{R}^n} h$, $\theta 
 \in \, ]0, 1[$ and $\beta \in \, ]0, \frac{1-\theta^{2}}{L}]$. Then the 
 se\-quen\-ce $\{x_k\}_{k}$, generated by the heavy ball method 
 \eqref{inertial:alg}, converges linearly to $\overline{x}$ and the sequence 
 $\{h(x_k)\}_{k}$ converges linearly to the optimal value $h^{*} = 
 h(\overline{x})$.
\end{theorem}

\begin{proof}
 Since $\nabla h$ is Lipschitz continuous, we have
 \begin{eqnarray}\label{descentproperty}
  h(x_{k+1}) & \le &  h(x_k) +\langle \nabla h(x_k), x_{k+1} -x_k
  \rangle +\frac{L}{2} \|x_{k+1} -x_k\|^2.
 \end{eqnarray}
 From \eqref{inertial:alg}, we can substitute $x_{k+1}-x_k$  by
 $\theta (x_{k} - x_{k-1}) - \beta \nabla h(x_{k})$ into
 \eqref{descentproperty} to deduce
 \begin{eqnarray}\label{descentproperty2}
   h(x_{k+1}) \nonumber
   &\le&  h(x_k) +\langle \nabla h(x_k), \theta (x_{k} - x_{k-1}) - \beta \nabla h(x_{k})\rangle\\ \nonumber
   && \quad + \frac{L}{2} \|\theta (x_{k} - x_{k-1}) - \beta \nabla h(x_{k})\|^2\\ \nonumber
   &=& h(x_k) - \beta \left (1-\frac{\beta L}{2} \right )\|\nabla h(x_{k})\|^2
   +\frac{L\theta^2}{2} \| x_{k} - x_{k-1}\|^2 \\
   &&+ \theta(1-\beta L)\langle \nabla h(x_k), x_{k} - x_{k-1} \rangle.
\end{eqnarray}
We also have
\begin{eqnarray*}
\| x_{k+1} - x_{k}\|^2
&=& \|\theta (x_{k} - x_{k-1}) - \beta \nabla h(x_{k}) \|^2\\
&=& \theta^2 \|x_{k} - x_{k-1}\|^2 + \beta^2 \|\nabla h(x_{k})\|^2
-2 \theta \beta\langle \nabla h(x_k), x_{k} - x_{k-1} \rangle
\end{eqnarray*}
Multiplying both sides of the above equality with $\frac{1-\beta L}{2\beta}$
and adding the results to \eqref{descentproperty2}, we obtain
$$
h(x_{k+1}) + \frac{1-\beta L}{2\beta} \| x_{k+1} - x_{k}\|^2  \leq h(x_k) - \frac{\beta}{2}\|\nabla h(x_{k})\|^2
   +\frac{\theta^2}{2 \beta} \| x_{k} - x_{k-1}\|^2.
$$
Defining the energy function
$$
E_k := h(x_{k}) - h^* + \frac{\theta^2}{2\beta} \|x_{k} -x_{k-1}\|^2
$$
we can write the last inequality as
\begin{eqnarray}
    E_{k+1} \nonumber
    &\leq& E_k-  \frac{\beta}{2}\|\nabla h(x_{k})\|^2 - \left( \frac{1-\beta L-\theta^2}{2\beta} \right)\|x_{k+1} -x_k\|^2 \\ \label{energyEstimate}
 &\leq& E_k-  \rho (\|\nabla h(x_{k})\|^2 + \|x_{k+1} -x_k\|^2),
\end{eqnarray}
where $\rho = \min\{ \frac{\beta}{2}, \frac{1-\beta L-\theta^2}{2
\beta} \} > 0$. \\

It remains to  bound $E_k$ from $\|\nabla h(x_{k})\|^2 + \|x_{k+1}
-x_k\|^2$. It follows from Proposition \ref{prop:PL} that
\begin{equation} \label{firstEst}
 h(x_k) - h^* \leq  \frac{2L}{\gamma^2}\|\nabla h(x_k)\|^2.
\end{equation}

%On the one hand, it follows from the Cauchy-Schwarz inequality
%and \eqref{gen:char} that
%$$
%\|\nabla h(x_k)\|\|x_k - \overline{x}\| \ge \langle \nabla h(x_k), x_{k} -\overline{x} \rangle \ge \frac{\gamma}{2} \|x_{k} -\overline{x}\|^2,
%$$
%which implies
%\begin{equation} \label{linearIne}
%    \|\nabla h(x_k)\| \ge \frac{\gamma}{2} \|x_{k} -\overline{x}\|,
%\end{equation}
%or equivalently
%$$
%\|\nabla h(x_k)\|^2 \ge \frac{\gamma^2}{4} \|x_{k} -\overline{x}\|^2.
%$$
%Combining this inequality with \eqref{descent:lemma}, remembering that $\nabla h(\overline{x}) = 0$, we obtain
%$$
%\|\nabla h(x_k)\|^2 \ge \frac{\gamma^2}{4} \|x_{k} -\overline{x}\|^2 \ge \frac{\gamma^2}{2L} (h(x_k) - h^*).
%$$
%i.e.
%\begin{equation} \label{firstEst}
%     h(x_k) - h^* \leq  \frac{2L}{\gamma^2}\|\nabla h(x_k)\|^2.
%\end{equation}
\noindent
On the other hand, using Cauchy-Schwarz inequality again, we have
\begin{eqnarray}
    \frac{1}{2\beta} \theta^2  \|x_{k} -x_{k-1}\|^2 \nonumber
    &= & \frac{1}{2\beta} \|x_{k+1} -x_{k} + \beta \nabla h(x_k)\|^2\\
    &\le& \frac{1}{\beta} \|x_{k+1} -x_{k} \|^2 + \beta \|\nabla h(x_k)\|^2. \label{secondEst}
\end{eqnarray}
Adding \eqref{firstEst} and \eqref{secondEst} we deduce
\begin{eqnarray*}
    E_k &\le& \left( \frac{2L}{\gamma^2} +  \beta\right) \|\nabla h(x_k)\|^2 + \frac{1}{\beta} \|x_{k+1} -x_{k} \|^2 \\
    &\le& \sigma  \left( \|\nabla h(x_k)\|^2 +  \|x_{k+1} -x_{k} \|^2 \right),
\end{eqnarray*}
where $\sigma = \max \{ \frac{2L}{\gamma^2} +   \beta, \frac{1}{\beta} \} >0$. \\
\noindent
Finally, combining the last inequality with \eqref{energyEstimate}
we deduce
$$ E_{k+1} \le \left (1-\frac{\rho}{\sigma} \right) E_k  \leq
 \left(1-\frac{\rho}{\sigma} \right)^k E_1,$$
which means that the sequence $\{E_k\}_k$ converges linearly to $0$. As a consequence,  the sequence
$\{h(x_k)\}_k$  converges linearly to the optimal value
 $h^*$ and the sequence $\{\|x_{k+1} - x_k\| \}_k$  converges linearly to $0$.
To deduce the linear convergence of $\{x_k\}_k$ to $\overline{x}$, we note from \eqref{inertial:alg} that
$$
\beta \|\nabla h(x_k)\| \le \|x_{k+1} - x_k \| + \theta \|x_{k+1} - x_k \|,
$$
which implies also that $\{\|\nabla h(x_k)\| \}_k$  converges linearly to $0$. Hence, it follows from \eqref{linearIne}
 that  $\{x_k\}_k$ converges linearly to $\overline{x}$.
\end{proof}

\begin{corollary}[Convergence rate]
 Under the assumptions of Theorem \ref{convergenceHeavyball}, we have
 the following convergence rates
 \begin{align}
  & h(x_{k+1}) - h^{*} \leq \left(1 - \frac{\rho}{\sigma} \right)^k E_{1}, \notag\\
  & \|x_{k+1} - x_{k}\|^2 \leq \frac{2\beta}{\theta^2}
  \left(1-\frac{\rho}{\sigma} \right)^k E_1.\notag
 \end{align}
 Moreover, for all $k \in \mathbb{N}$, we have
 \begin{align*}
  \|\nabla h(x_{k}) \| & \leq \frac{(1+\theta)}{\theta} \left(1 -
  \frac{\rho}{\sigma} \right)^{\frac{k-1}{2}}  \sqrt{\frac{2}{\beta}}
  \sqrt{E_1}, \\[2mm]
  \|x_k - \overline{x}\| & \le \frac{2(1+\theta)}{\gamma \,\theta}
  \left(1-\frac{\rho}{\sigma} \right)^{\frac{k-1}{2}} \sqrt{\frac{2}{\beta}}
  \sqrt{E_1},
 \end{align*}
 where %$\beta \in \, ]0, \frac{1}{L}[$, $\theta \in \, ]0, \sqrt{1 - \beta L}[$,
 %{\color{red}
 $\theta \in \, ]0, 1[$, $\beta \in \, ]0,
 \frac{1-\theta^{2}}{L}]$,
 %},
 $\rho = \min\{ \frac{\beta}{2}, \frac{1-\beta L-\theta^2}{2 \beta} \}$,
 $\sigma= \max\{ \frac{2L}{\gamma^2} +   \beta, \frac{1}{\beta} \}$ and
 $E_1:=h(x_0) -h^{*}+ \frac{\theta^2}{2\beta}\|x_1- x_0\|^2$.
\end{corollary}

\begin{remark}
 A second-order differential equation without viscous damping (i.e. $\alpha 
 = 0$) and gradient ascent (i.e. $\ddot{x}(t) = \nabla h (x(t))$) was recently 
 considered in \cite{PK} for (strongly) quasiconvex functions. The convergence 
 results for continuous and discrete time provided there were obtained under stronger 
 conditions.
\end{remark}

\section{Conclusions}

We contributed to the discussion on strongly quasiconvex functions by studying the differentiable case. In particular, we developed new properties for differentiable strongly quasiconvex functions that were used for studying gradient dynamics in the first and second order as well as their discretizations for the gradient method and the heavy ball acceleration.

We hope that our contribution could provide new lights for further developments on accelerated versions of first order algorithms for generalized convex problems and, in particular, for strongly quasiconvex functions.

\section{Declarations}

%\subsection{Ethical Approval and Consent to participate}

%Not Applicable.

%\subsection{Consent for publication}

%Not Applicable.

%\subsection{Human and Animal Ethics}

%Not Applicable.

\subsection{Availability of supporting data}

No data sets were generated during the current study. 

\subsection{Author Contributions}

 All authors contributed equally to the study conception, design and implementation
 and wrote and corrected the manuscript.

\subsection{Competing Interests}

There are no conflicts of interest or competing interests related to this manuscript.

\subsection{Funding}

 This research was partially supported by ANID--Chile through Fondecyt 
 Regular 1241040 (Lara) and by BASAL fund FB210005 for center of excellence from ANID-Chile (Marcavillaca).

%\subsection{Acknowledgements}
% We thank S.-M. Grad for references \cite{JO3,K-1980} and A. Gasnikov for reference \cite{VNC-2} (in Russian), respectively.


\begin{thebibliography}{99}

\bibitem{Ant}
 \textsc{A.S. Antipin}, Minimization of convex functions on convex sets
 by means of differential equations, {\it Diff. Equations}, \textbf{30},
1365--1375, (1994).

%\bibitem{ACFR}
% \textsc{H. Attouch, Z. Chbani, J. Fadili, H. Riahi}, First-order optimization
% algorithms via inertial systems with Hessian driven damping, {\it Math.
% Programm.}, 1--43, (2022).

\bibitem{ACPR}
 {\sc H. Attouch, Z. Chbani, J. Peypouquet, P. Redont}, Fast convergence
 of inertial dynamics and algorithms with asymptotic vanishing viscosity,
 {\it Math. Programm.}, {\bf 168}, 123--175, (2018).

\bibitem{AC2017}
 \textsc{J.-F. Aujol, Ch. Dossal},
 Optimal rate of convergence of an ODE associated to the fast gradient
 descent schemes for $b > 0$, {\it Hal Preprint hal-01547251}, (2017).

\bibitem{ADR}
 \textsc{J.-F. Aujol, Ch. Dossal, A. Rondepierre}, Convergence rate of the
 heavy ball method for quasi-strongly convex optimization, {\it SIAM J.
 Optim.}, \textbf{32(2)}, 1817--1842, (2022).

\bibitem{ADR2}
 \textsc{J.-F. Aujol, Ch. Dossal, A. Rondepierre}, Optimal convergence
  rates for Nesterov acceleration, {\it SIAM J. Optim.}, \textbf{29(4)},
 3131--3153, (2019).

\bibitem{AE}
 \textsc{K.J. Arrow, A.C. Enthoven}, Quasiconcave programming, {\it
 Econo\-me\-tri\-ca}, \textbf{29}, 779--800, (1961).

\bibitem{ADSZ}
 {\sc M. Avriel, W.E. Diewert, S. Schaible, I. Zang}, ``Generalized
 Concavity''. SIAM, Philadelphia, (2010).

\bibitem{BS}
 {\sc R.I Bo\c{t}, E.R. Csetnek}, Convergence rates for forward--backward
 dy\-na\-mi\-cal systems associated with strongly monotone inclusions,
 {\it J. Math. Anal. Appl.}, \textbf{457}, 1135--1152, (2018).

\bibitem{Bolte}
 {\sc J. Bolte}, Continuous gradient projection method in Hilbert spaces, 
 {\it J. Optim. Theory Appl.}, \textbf{119(2)}, 235--259, (2003).

\bibitem{CEG}
{\sc A. Cabot, H. Engler, S. Gadat}, On the long time behavior of second
 order differential equations with asymptotically small dissipation, {\it Trans.
 Amer. Math. Soc.}, {\bf 361}, 5983--6017 (2009).

\bibitem{CM-Book}
 {\sc A. Cambini, L. Martein}. ``Generalized Convexity and Optimization:
 Theo\-ry and Applications''. Springer, (2009).

%\bibitem{C}
% {\sc J.P. Crouzeix}, Characterizations of generalized convexity and
% generalized monotonicity, a survey. In: J.P. Crouzeix et al. (eds.):
% Generalized Con\-ve\-xi\-ty, Generalized Monotonicity. pp. 237--256,
% Kluwer Academics Publishers, (1998).

\bibitem{GOM}
 \textsc{X. Goudou, J. Munier},  The gradient and heavy ball with friction 
 dynamical systems: the quasiconvex case, {\it Math. Programm.}, {\bf 116}, 
 173--191, (2009).

\bibitem{GLM-1}
 {\sc S.-M. Grad, F. Lara, R.T. Marcavillaca}, Relaxed-inertial proximal point type algorithms for quasiconvex minimization, {\it J. of Global Optim.}, {\bf 85}, 3, 615--635, (2023).

\bibitem{HKS}
 {\sc N. Hadjisavvas, S. Komlosi, S. Schaible}, ``Handbook of Generalized 
 Convexity and Generalized Monotonicity''. Springer-Verlag, Boston, (2005).

\bibitem{ILMY}
 {\sc A. Iusem, F. Lara, R.T. Marcavillaca, L.H. Yen}, A two-step pro\-xi\-mal 
 point algorithm for nonconvex equilibrium problems with applications to 
 fractional programming. {\it J. Global Optim.}, DOI: 10.1007/s10898-024-01419-8, (2024).

\bibitem{JO3}
\textsc{M. Jovanovi\'c}, Strongly quasiconvex quadratic functions, {\it Publ. Inst. Math., Nouv. S\'er.}, {\bf 53}, 153--156, (1993).

\bibitem{KL-2} 
{\sc A. Kabgani, F. Lara}, Strong subdifferentials: theory and
applications in nonconvex optimization, {\it J. Global Optim.}, {\bf 84},
 349--368, (2022).

\bibitem{Ka-Nu-Sc}
 {\sc H. Karimi, J. Nutini, M. Schmidt}, Linear convergence of gradient and
 proximal-gradient methods under the Polyak-{\L}ojasiewicz condition.
 In ``Machine Learning and Knowledge Discovery in Databases", pages 
 795--811. Springer, (2016).

\bibitem{KK}
 {\sc G. Kassay, J. Kolumb\'an}, Multivalued parametric variational
 inequa\-li\-ties with $\alpha$-pseudomonotone maps, {\it J. Optim.
 Theory Appl.}, {\bf 107}, 35--50, (2000).

 \bibitem{K-1980}
 {\sc A.I. Korablev}, Relaxation methods of minimization of pseudoconvex functions, {\it J. Soviet Math.}, {\bf 44}, 1--5, (1989) (translated from {\it Issled Prikl Mat.}, {\bf 8}, 3--8, (1980)).

\bibitem{Lara-9}
 {\sc F. Lara}, On strongly quasiconvex functions: existence results and 
 proximal point algorithms, {\it J. Optim. Theory Appl.}, {\bf 192}, 891--911, (2022).

\bibitem{LM}
{\sc F. Lara, R.T. Marcavillaca}, Bregman proximal point type algorithms for quasiconvex minimization, {\it Optimization}, {\bf 73}, 497--515, (2024).

\bibitem{L}
 {\sc S. {\L}ojasiewicz}, A topological property of real analytic subsets.
 Coll. du CNRS, Les \'equations aux d\'eriv\'ees partielles, {\bf 117},
 87--89, (1963).

\bibitem{MM}
 {\sc Y. Malitsky, K. Mishchenko}, Adaptive gradient descent without
 descent, {\it Proc. 37th Int. Conf. Mach. Learning, PMLR 119}, (2020).

\bibitem{NP}
 {\sc I. Necoar\u{a}, Y. Nesterov, F. Glineur}, Linear convergence of first-order 
 methods for non-strongly convex optimization, {\it Math. Program.},
 {\bf 175}, 69--107, (2019).

\bibitem{Nesterov-book}
 {\sc Y. Nesterov}, `` Lectures on convex optimization''. Springer, Berlin, 
 (2018).

\bibitem{P1}
 {\sc B.T. Polyak}, Gradient methods for minimizing functionals, {\it Zh.
 Vychisl. Math. Mat. Fiz.}, {\bf 3}, 643--653, (1963).

\bibitem{P2}
 {\sc B.T. Polyak}, Some methods of speeding up the convergence of
  iteration methods, {\it USSR Comp. Math. and Math. Phys.}, {\bf 4(5)},
  1--17, (1964).

\bibitem{P}
 \textsc{B.T. Polyak}, Existence theorems and convergence of minimizing
 sequences in extremum problems with restrictions, \textit{Soviet Math.},
 \textbf{7}, 72--75, (1966).

\bibitem{PS}
 \textsc{B.T. Polyak, P. Shcherbakov}, Lyapunov functions: An optimization
 theory perspective, \textit{IFAC-PapersOnLine}, {\bf 50}, 7456--7461,
 (2017).

\bibitem{PK}
 {\sc M. Rahimi Piranfar, H. Khatibzadeh}, Long-Time behavior of a
 gradient system governed by a quasiconvex function, {\it J. Optim. Theory 
 Appl.}, {\bf 188}, 169--191, (2021).

\bibitem{rock-1980}
 {\sc R.T. Rockafellar}, Generalized directional derivatives and
 subgradients of nonconvex functions, {\it Can. J. Math.}, {\bf 32},
 257--280, (1980).

%\bibitem{V2}
% {\sc J.P. Vial}, Strong convexity of sets and functions, {\it J. Math. Econ.},
% {\bf 9}, 187--205, (1982).

\bibitem{VNC-2}
 {\sc A.A. Vladimirov, Ju.E. Nesterov, Ju.N. Chekanov}, O ravnomerno 
 kvazivypuklyh funkcionalah [On uniformly quasiconvex functionals], {\it 
 Vestn. Mosk. un-ta, vycis. mat. i kibern.}, {\bf 4}, 18--27, (1978) (In Russian). 

\bibitem{Vu}
 {\sc P.T. Vuong}, A second-order dynamical system and its discretization
 for strongly pseudo-monotone variational inequality, {\it SIAM J. Control
 Optim.}, {\bf 59}, 2875--2897, (2021).

\end{thebibliography}
\end{document}